\documentclass{amsart}

\usepackage{tikz}
\usetikzlibrary{arrows.meta}
\usepackage{commutative-diagrams}
\usetikzlibrary{cd}
\usepackage{geometry}
\usepackage[colorlinks=true]{hyperref}
\usepackage[all]{hypcap}
\usepackage{mathtools}
\usepackage{booktabs}
\usepackage{tabularx}
\usepackage{array}
\usepackage{amssymb}
\usepackage{dsfont}
\usepackage{witharrows}
\usepackage{threeparttable}

\newtheorem{theorem}{Theorem}[section]
\newtheorem{lemma}[theorem]{Lemma}

\newtheorem{proposition}[theorem]{Proposition}

\theoremstyle{definition}

\newtheorem{example}[theorem]{Example}

\theoremstyle{remark}
\newtheorem{remark}[theorem]{Remark}

\newcommand{\Mod}[1]{\ (\mathrm{mod}\ #1)}

\numberwithin{equation}{section}

\newcolumntype{C}{>{\centering\arraybackslash}X}

\begin{document}

\title[Induced subgraphs with degree parity in Paley graphs]{On induced subgraphs with degree parity conditions in Paley graphs and Paley tournaments}


\author{Qilong Li}
\address{College of Science, National University of Defense Technology, 410073 Changsha, China}
\email{li.qilong@outlook.com}
\thanks{}
\author{Yue Zhou}
\address{College of Science, National University of Defense Technology, 410073 Changsha, China}
\email{yue.zhou.ovgu@gmail.com}
\thanks{}



\begin{abstract}
	In this paper, we investigate the number of induced subgraphs and subdigraphs of Paley graphs and Paley tournaments where the (out-)degree of each vertex has the same parity. For Paley graphs, we establish a lower bound for the number of large even induced subgraphs, particularly those containing a constant proportion of vertices. We determine the number of even-even partitions of Paley graphs, showing it is exponential if $q\equiv 1\Mod{8}$ and is trivial if $q\equiv 5\Mod{8}$, while proving the non-existence of even-even partition for Paley tournaments. Furthermore, we derive asymptotic formulas for the numbers of even induced sub(di)graphs of order $r=o(q^{1/4})$ in Paley graphs and Paley tournaments, demonstrating their concentration around the expected values in the corresponding random (di)graph models.\par 
	In the context of coding theory, we establish a correspondence between even/odd induced sub(di)graphs of Paley graphs (tournaments) and maximum distance separable (MDS) self-dual codes that can be constructed via (extended) generalized Reed-Solomon codes from subsets of finite fields. As a consequence, our contribution on induced subgraphs leads to new existence and counting results about MDS self-dual codes.
	
\end{abstract}

\maketitle

\section{Introduction}
Let $q$ be an odd prime power and let $\mathbb{F}_{q}$ be the finite field with $q$ elements. If $q\equiv 1\Mod{4}$, we define the \emph{Paley graph} $P_{q}$ on vertex set $\mathbb{F}_{q}$ by joining two vertices $x$ and $y$ with an edge if and only if $x-y\in\square_{q}$, where $\square_{q}$ stands for all square elements in $\mathbb{F}_{q}^*$. If $q\equiv 3\Mod{4}$, then the \emph{Paley tournament} $PT_{q}$ of order $q$ is the digraph on vertex set $\mathbb{F}_{q}$ where $(x,y)$ is an arc if and only if $x-y\in\square_{q}$; this is a well-defined tournament since $-1\notin\square_{q}$. Behaving in many ways like random (di)graphs with edge probability $p=1/2$, the Paley graphs (tournaments) are a famous family of \emph{quasi-random graphs (tournaments)}, which are commonly seen as explicit constructions of the random structures. The quasi-randomness of graphs and tournaments on $n$ vertices are rigorously defined by a sequence of conditions that are all equivalent \cite{chung1988quasi-random,chung1991quasi-random}, among which that the second-largest eigenvalue is $o(n)$ is sometimes the easiest to be verified. We refer to the book \cite{alon2016probabilistic} for more details.\par 
In the literature, various particular induced subgraphs of Paley graphs $P_{q}$ have been investigated, for example, the chromatic number and the Tur\'{a}n number of $P_{q}$ was studied in \cite{broere1988clique} and \cite{fox2025largest}, respectively. The cliques of $P_{q}$ have also been studied extensively \cite{cohen1988clique,baker1996maximal,bachoc2013squares,goryainov2018eigenfunctions,hanson2021refined,goryainov2022correspondence,yip2022maximal,yip2022clique}, partially due to their role in the determination of Ramsey numbers \cite{greenwood1955combinatorial,evans1981number,mathon1987lower,atanasov2014certain}.\par 
The existence of any small induced subgraph of constant order in $P_{q}$ and $PT_{q}$ is guaranteed by the definition of quasi-randomness, namely, for every quasi-random graph (resp. tournament) $G$ on $n$ vertices, the asymptotic number $\text{ind}(H,G)$ of labeled occurrences of any graph (resp. tournament) $H$ of constant order $r$ in $G$ is given by
\begin{equation*}
	\text{ind}(H,G)=\left(1+o(1)\right)n^{r}2^{-\binom{r}{2}}.
\end{equation*}
For all integers $r=r(q)\geq 3$, it was proved that $P_{q}$ contains any graph of order $r$ as an induced subgraph if $q$ is roughly $r^{2}4^{r}$ \cite[Theorem 3]{bollobas1981graphs}. More precisely, we have the following result.
\begin{theorem}[{\cite[Theorem 13.11]{bollobas2001random}}]\label{thm_paley.graph.r.full}
	Let $r\geq 3$ and let $q\equiv 1\Mod{4}$ be a prime power satisfying
	\begin{equation}\label{ineq_paley.graph.r.full.bound}
		q>\left((r-3)2^{r-2}+1\right)\sqrt{q}+(r-1)2^{r-2}.
	\end{equation}
	Then $P_{q}$ contains every graph of order $r$ as an induced subgraph.
\end{theorem}
In this paper, we direct our focus on \emph{even} and \emph{odd induced sub(di)graphs} of Paley graphs and Paley tournaments. We say that an induced subgraph $H$ of a graph $G$ is even (resp. odd) if the degree of each vertex $v\in V(H)$ in $H$, denoted by $\deg_{H}(v)$, is even (resp. odd). Even (resp. odd) induced subdigraphs $F$ of a digraph $D$ are defined likewise by considering the parity of the out-degree $\deg_{F}^{-}(v)$ of each vertex $v\in V(F)$.\par 
For any undirected graph $G$, a classical theorem of Gallai (see Theorem \ref{thm_gallai.even.partition}) shows that there exists a partition of $V(G)$ such that each part induces an even subgraph, ensuring the existence of even induced subgraphs of $G$ with at least $|V(G)|/2$ vertices. We call such partition $V(G)=V_{1}\sqcup V_{2}$ an \emph{even-even partition} and call $G[V_{1}]$ and $G[V_{2}]$ a pair of \emph{co-even induced subgraphs} of $G$. Assume that $|V_{1}|\geq|V(G)|/2$, then we say that $G[V_{1}]$ is a \emph{primary co-even induced subgraph} of $G$.\par 
Recently, the number of such partition is determined asymptotically for the random graphs with more general degree congruence conditions \cite{balister2023counting}. For the odd induced subgraphs, it was proved by Scott in \cite[Theorem 1]{scott2001induced} that $V(G)$ has a partition $V_{1},V_{2},...,V_{k}$ such that $G[V_{i}]$ are all odd induced subgraphs if and only if $|V(G)|$ is even. A famous conjecture, which was described by Caro \cite{caro1994induced} as “part of the graph theory folklore”, goes that for all graphs $G$ (without isolated vertices) of order $n$, there exists an absolute constant $c>0$ such that the largest order of the odd induced subgraph of $G$ is at least $cn$. This was ultimately confirmed by Ferber and Krivelevich \cite{ferber2022every} in 2022, showing the statement holds with constant $c=1/10000$.\par 
Nevertheless, for Paley graphs $P_{q}$, Gallai's theorem only ensures the existence of its primary even induced subgraph, which would be trivial since $P_{q}$ itself has all degrees even. In Section \ref{sect_linearly.sized.gallai.even.induced.subgraphs}, we concentrate on the number of the \emph{giant} even induced subgraph of arbitrary undirected simple graph $G$, viz.\,the induced subgraph that contains a constant proportion of vertices in $G$. As a consequence, we have the following theorem.
\begin{theorem}\label{thm_number.of.linear.size.even.induced.subgraph.lower.bound}
	Let $G$ be any undirected simple graph on $n$ vertices and let $0<\alpha<1$ be a constant. Then the number of even induced subgraphs of $G$ with order lying in the interval $[\alpha n/2,\alpha n]$ is at least
	\begin{equation*}
		2^{\left(1+o(1)\right)\left(H(\frac{\alpha}{2})-\alpha\right)n},
	\end{equation*}
	where $H(x)=-x\log_{2}{x}-(1-x)\log_{2}{(1-x)}$ for $0<x<1$ is the binary entropy function.
\end{theorem}
From this result, one can easily obtain a lower bound for the number of giant even induced subgraphs of $P_{q}$ by considering the graphs $G$ induced by any (constant) proportion of vertices in $P_{q}$. In Section \ref{sect_gallai.even.induced.subgraphs.paley} we determine the number of all co-even induced sub(di)graphs of $P_{q}$ by proving the following theorem.
\begin{theorem}\label{thm_number.of.all.co.even.induced.subgraphs.paley}
	Let $N$ denote the number of co-even induced subgraphs of $P_{q}$. Then
	\begin{equation*}
		N=\begin{dcases}
			2^{\frac{q+1}{2}} & \text{if $q\equiv 1\Mod{8}$},\\
			2 & \text{if $q\equiv 5\Mod{8}$}.
		\end{dcases}
	\end{equation*}
	In particular, if $q\equiv 5\Mod{8}$, then there exist only trivial co-even induced subgraphs of $P_{q}$, namely the empty graph and $P_{q}$ itself.
\end{theorem}
For digraphs, we define the \emph{even-even partition} and \emph{co-even induced subdigraphs} likewise to the undirected graphs. It turns out that for any tournament $T$, an even induced subdigraph of order $r$ in $T$ exists only if $r\equiv 0,1\Mod{4}$ (see Lemma \ref{lm_non.existence.co.even.induced.subgraphs.paley.directed}), therefore there does not exist co-even induced subgraph of $PT_{q}$.\par 
In Section \ref{sect_even.induced.subgraphs.paley.small}, asymptotic formulas for the number of even induced sub(di)graphs of order $r=r(q)$ will be given for both $P_{q}$ and $PT_{q}$ when $r=o(q^{1/4})$, showing that for $r=o(q^{1/4})$ the numbers of even induced sub(di)graphs of order $r$ concentrate around the expectations $2^{1-r}\binom{q}{r}$ and $2^{-r}\binom{q}{r}$ in the random (di)graphs (see Appendix \ref{appendix.a}), respectively. Here we use the Landau asymptotic notation, namely for real-valued functions $g(x)$ and $f(x)$, where $f(x)$ is positive, we write $g(x)=O(f(x))$ if there exists constant $C>0$ such that $|g(x)|\leq Cf(x)$ for all sufficiently large $x$; meanwhile, by $g(x)=o(f(x))$, or equivalently $f(x)=\omega(g(x))$, we mean $\lim_{x\to\infty}g(x)/f(x)=0$.
\begin{theorem}\label{thm_number.of.even.induced.subgraphs.paley.undirected}
	Let $N_{r}$ denote the number of even induced subgraphs of order $r$ in $P_{q}$. If $r=\omega(1)$ and $r=o(q^{1/4})$, then
	\begin{equation*}
		N_{r}=\left(1+o(1)\right)2^{1-r}\binom{q}{r}.
	\end{equation*}
\end{theorem}
\begin{theorem}\label{thm_number.of.even.induced.subgraphs.paley.directed}
	Let $\tilde{N}_{r}$ denote the number of even induced subdigraphs of order $r$ in $PT_{q}$. If $r=\omega(1)$, $r=o(q^{1/4})$ and $r\equiv 0,1\Mod{4}$, then
	\begin{equation*}
		\tilde{N}_{r}=\left(1+o(1)\right)2^{1-r}\binom{q}{r}.
	\end{equation*}
\end{theorem}
As an application of these results, we find a correspondence between the sub(di)graph and a large family of MDS self-dual codes which are constructed via (extended) GRS codes from subsets of finite fields. The construction of these codes has been a thriving research area in the past decade, with numerous explicit constructions based on the arithmetic properties of the subsets of $\mathbb{F}_{q}$ (e.g., \cite{jin2017new,yan2019note,fang2019new,lebed2022some,wan2023newmds,fang2026construction}). Wide ranges of possible length $n$ of MDS self-dual codes (which corresponds to the order of induced sub(di)graphs) were provided, while restrictions on the field size $q$ and other parameters are imposed and there is a lack of global existence or enumeration results. Our contribution on the induced sub(di)graphs leads to new existence and counting results on MDS self-dual codes, which are summarized in Section \ref{sect_applications.to.mds.self.dual}.\par 
Finally, we note that for induced subgraphs with the degree parity conditions in Paley graph $P_{q}$, studying the even induced subgraphs will suffice. This is because an odd induced subgraph $H$ of $P_{q}$ must be of even order, hence the subgraph of the complement of $P_{q}$ induced by $V(H)$ must have all degrees even. Therefore the number of odd induced subgraphs of $P_{q}$ equals the number of the even induced subgraphs of $P_{q}$ of even order, given that $P_{q}$ is self-complementary. Meanwhile, a similar argument applies to the odd induced subdigraphs of even order in $PT_{q}$, noting that $PT_{q}$ is self-converse.

\section{Preliminaries}\label{sect_preliminaries}
\subsection{Even-even partition and induced subgraphs with degree parity conditions}
The even and odd induced subgraphs in arbitrary graphs have been studied extensively. The following theorem guarantees the existence for even-even partition in any graph, see \cite[Problem 5.17]{lovasz2007combinatorial} for a proof.
\begin{theorem}[Gallai's Theorem]\label{thm_gallai.even.partition}
	Let $G$ be any graph. Then there exists a partition $V(G)=V_{1}\sqcup V_{2}$ such that both $G[V_{1}]$ and $G[V_{2}]$ are even induced subgraphs of $G$.
\end{theorem}
An interpretation of Theorem \ref{thm_gallai.even.partition} using linear algebra was given by Caro \cite{caro1996simple}. Let $G$ be any undirected simple graph (i.e. without loops and multiple edges). Let $N_{G}[v]$ denote the closed neighborhood of $v$ for each vertex $v\in V(G)$, namely the set $\left\{w\in V(G)\mid v\sim w\right\}\cup\left\{v\right\}$. A subset $Q\subseteq V(G)$ is called an \emph{odd-parity cover} if for every vertex $v\in V(G)$, $|N_{G}[v]\cap Q|\equiv 1\Mod{2}$. Let $A$ denote the adjacency matrix of $G$ and write the characteristic vector of vertex set $Q$ as $\boldsymbol{1}_{Q}$. Then it is straightforward to see that $Q$ is an odd-parity cover if and only if $(A+I)\boldsymbol{1}_{Q}=\boldsymbol{1}$ over $\mathbb{F}_{2}$, where $\boldsymbol{1}$ stands for the all-ones vector.\par 
In the proof of \cite[Theorem 6]{caro1996simple} the following result was given, where the \emph{odd-extension} of graph $G$, written $\overline{G}$, is defined to be the graph formed by attaching every vertex $v\in V(G)$ of even degree a new vertex $v'$ that is adjacent only to $v$.
\begin{theorem}[\cite{caro1996simple}]\label{thm_caro.odd.parity.cover.and.gallai.partition}
	Let $G$ be any graph and suppose that $W$ is an odd-parity cover of $\overline{G}$. Then the set $W\cap V(G)$ and its complement in $V(G)$ form an even-even partition of $G$.
\end{theorem}
In $\S$\ref{subsect_characterization.co-even.induced.subgraphs}, we will furthermore prove that there actually exists a one-to-one correspondence between even-even partitions of $G$ and odd-parity covers of $\overline{G}$. Hence, the problem of enumerating co-even induced subgraphs of any undirected graph $G$ is equivalent to determining the solution space of equation $(\overline{A}+I)X=\boldsymbol{1}$ over $\mathbb{F}_{2}$, where $\overline{A}$ stands for the adjacency matrix of $\overline{G}$.\par 
Let $T$ be any tournament. Then for the existence of even and odd induced subdigraphs of $T$ we have the following result.
\begin{lemma}\label{lm_non.existence.co.even.induced.subgraphs.paley.directed}
	Let $S$ be any subset of $V(T)$. Then $T[S]$ is an even induced subdigraph of $T$ only if $|S|\equiv 0,1\Mod{4}$ and $T[S]$ is an odd induced subdigraph only if $|S|\equiv 0,3\Mod{4}$.
\end{lemma}
The proof of Lemma \ref{lm_non.existence.co.even.induced.subgraphs.paley.directed} is straightforward, since for each induced subdigraph $T[S]$ we have $|S|(|S|-1)/2=\sum_{v\in S}\deg_{T[S]}^{-}(v)$, the parity of which is even when every $\deg_{T[S]}^{-}(v)$ is even, and is the same as the parity of $|S|$ when every $\deg_{T[S]}^{-}(v)$ is odd.
\subsection{Motivation from coding theory and related problems}\label{subsect_motivation.from.coding.theory}
For the prime power $q$, a $q$-ary $[n,k,d]$-linear code $\mathcal{C}$ is a $k$-dimensional subspace of $\mathbb{F}_{q}^{n}$ with minimum Hamming distance $d$. It is well-known that parameters $n,k$ and $d$ must be subject to the \emph{Singleton bound} $d\leq n-k+1$. The code $\mathcal{C}$ is called a \emph{maximum distance separable code}, or simply an \emph{MDS code} if the equality is attained. For two vectors $u=(u_{1},u_{2},...,u_{n})$ and $w=(w_{1},w_{2},...,w_{n})$ in $\mathbb{F}_{q}^{n}$, the \emph{Euclidean inner product} of $u$ and $w$ is defined by $\langle u,w\rangle=\sum_{i=1}^{n}u_{i}w_{i}\in\mathbb{F}_{q}$. The \emph{(Euclidean) dual code} $\mathcal{C}^{\perp}$ of $\mathcal{C}$ is then defined by
\begin{equation*}
	\mathcal{C}^{\perp}:=\left\{u\in\mathbb{F}_{q}^{n}\mid\langle u,w\rangle=0~\text{for all $w\in\mathcal{C}$}\right\}.
\end{equation*}
If $\mathcal{C}=\mathcal{C}^{\perp}$, we say that $\mathcal{C}$ is \emph{(Euclidean) self-dual}. If $\mathcal{C}$ is both self-dual and MDS, we say that $\mathcal{C}$ is an \emph{MDS self-dual code}.\par 
It is clear that for each MDS self-dual $[n,k,d]$-linear code $\mathcal{C}$ we have $n=2k$ and therefore $d=n/2+1$, so the parameters of $\mathcal{C}$ are completely determined by its length $n$. The study of MDS self-dual codes usually focuses on the possible length $n$. This was completely solved by Grassl and Gulliver in \cite{grassl2008self-dual} if $q$ is even, who showed that in this case there exist MDS self-dual codes of length $n$ for every $n\leq q$; while the existence and the quantity of MDS self-dual codes over the finite fields with odd characteristic remain open.\par 
In the literature, one of the most commonly used tool to construct explicit families of MDS self-dual codes is the \emph{generalized Reed-Solomon codes}, or simply the \emph{GRS codes}, see for instance \cite{lebed2019construction,fang2019new,yan2019note,zhang2020unified,fang2020new}. We note that a GRS code over $\mathbb{F}_{q}$ associated with vectors $\boldsymbol{\alpha}=(\alpha_{1},\alpha_{2},...,\alpha_{n})$ and $\boldsymbol{v}=(v_{1},v_{2},...,v_{n})$ in $\mathbb{F}_{q}^{n}$, where the $\alpha_{i}$'s are distinct and the $v_{i}$'s are non-zero, is defined by
\begin{equation*}
	GRS_{k}(\boldsymbol{\alpha},\boldsymbol{v}):=\left\{(v_{1}f(\alpha_{1}),v_{2}f(\alpha_{2}),...,v_{n}f(\alpha_{n}))\mid f(x)\in\mathbb{F}_{q}[x],~\deg{(f(x))}\leq k-1\right\}.
\end{equation*}
It is well-known that the code $GRS_{k}(\boldsymbol{\alpha},\boldsymbol{v})$ is a $q$-ary $[n,k,n-k+1]$-MDS code and that its dual is also a GRS code. Furthermore, consider the \emph{extended GRS code} of $GRS_{k}(\boldsymbol{\alpha},\boldsymbol{v})$ given by
\begin{equation*}
	GRS_{k}(\boldsymbol{\alpha},\boldsymbol{v},\infty):=\left\{(v_{1}f(\alpha_{1}),v_{2}f(\alpha_{2}),...,v_{n}f(\alpha_{n}),f_{k-1})\mid f(x)\in\mathbb{F}_{q}[x],~\deg{(f(x))}\leq k-1\right\},
\end{equation*}
where $f_{k-1}$ stands for the coefficient of $x^{k-1}$ in $f(x)$. The code $GRS_{k}(\boldsymbol{\alpha},\boldsymbol{v},\infty)$ is a $q$-ary $[n+1,k,n-k+2]$-MDS code. The following result provides unified criteria for (extended) GRS codes to be MDS self-dual (see \cite[Theorem 1]{zhang2020unified}).
\begin{theorem}\label{thm_grs.self.dual.unified.criteria}
	Let $\mathcal{S}=\left\{\alpha_{1},\alpha_{2},...,\alpha_{n}\right\}$ be a non-empty subset of $\mathbb{F}_{q}$ and define
	\begin{equation*}
		\Delta_{\mathcal{S}}(\alpha_{i})=\prod_{1\leq j\leq n,j\neq i}(\alpha_{i}-\alpha_{j})\in\mathbb{F}_{q}^*
	\end{equation*}
	for each $1\leq i\leq n$. Then we have:\par 
	(i) (\cite{jin2017new}) Suppose $n$ is even. There exists $\boldsymbol{v}\in(\mathbb{F}_{q}^*)^{n}$ such that $GRS_{n/2}(\boldsymbol{\alpha},\boldsymbol{v})$ is self-dual if and only if all $\eta\left(\Delta_{\mathcal{S}}(\alpha_{i})\right)$, $1\leq i\leq n$ are the same;\par 
	(ii) (\cite{yan2019note}) Suppose $n$ is odd. There exists $\boldsymbol{v}\in(\mathbb{F}_{q}^*)^{n}$ such that $GRS_{(n+1)/2}(\boldsymbol{\alpha},\boldsymbol{v},\infty)$ is self-dual if and only if $\eta\left(-\Delta_{\mathcal{S}}(\alpha_{i})\right)=1$ for all $1\leq i\leq n$.
\end{theorem}
Observe that for each subset $\mathcal{S}\subseteq\mathbb{F}_{q}$ with $|\mathcal{S}|=n$, if $n$ is even, then $\eta(\Delta_{\mathcal{S}}(\alpha_{i}))=1$ for all $1\leq i\leq n$ if and only if $\mathcal{S}$ induces a sub(di)graph with all (out-)degrees even. Likewise, if $n$ is odd, then $\eta(\Delta_{\mathcal{S}}(\alpha_{i}))=1$ for all $1\leq i\leq n$ if and only if $\mathcal{S}$ induces sub(di)graph with all (out-)degrees odd, etc. Therefore, finding MDS self-dual codes over $\mathbb{F}_{q}$ that can be constructed via (extended) GRS codes is equivalent to finding the induced sub(di)graphs with the degree parity conditions in $P_{q}$ and $PT_{q}$. See Section \ref{sect_applications.to.mds.self.dual} for more details.
\subsection{Binomial coefficients}
We list several standard facts on the binomial coefficients that will be used, some of which can be found in \cite{das2016brief} and \cite[Section 1.1]{jukna2011extremal}. All the following asymptotic results can be derived from Stirling's formula. Note that we have inequalities
\begin{equation*}
	m\log{\frac{n}{m}}\leq\log{\binom{n}{m}}\leq m\left(\log{\frac{n}{m}}+1\right)
\end{equation*}
for all positive integers $m\leq n$.
\begin{lemma}\label{lemma_binomial.coefficient.constant}
	If $m$ is a constant, then
	\begin{equation*}
		\binom{n}{m}=\left(1+o(1)\right)\frac{n^{m}}{m!}.
	\end{equation*}
\end{lemma}
\begin{lemma}\label{lemma_binomial.coefficients.sqrt.size}
	If $m$ satisfies $m=\omega(1)$ and $m=o(\sqrt{n})$, then
	\begin{equation*}
		\binom{n}{m}=\left(1+o(1)\right)\frac{1}{\sqrt{2\pi m}}(\frac{\mathrm{e}n}{m})^{m}.
	\end{equation*}
\end{lemma}
\begin{lemma}\label{lemma_binomial.coefficients.larger.than.sqrt.size}
	If $m$ satisfies $m=\omega(1)$, $m=o(n)$ and $m\neq o(\sqrt{n})$, then
	\begin{equation*}
		\binom{n}{m}=\exp{\left\{\left(1+o(1)\right)m\log{\frac{n}{m}}\right\}}.
	\end{equation*}
\end{lemma}
\begin{lemma}\label{lemma_binomial.coefficients.linear.size}
	If $m=\Omega(n)$, then
	\begin{equation*}
		\binom{n}{m}=2^{\left(1+o(1)\right)H(\frac{m}{n})n},
	\end{equation*}
	where $H(x)=-x\log_{2}{x}-(1-x)\log_{2}{(1-x)}$ for $0<x<1$ is the binary entropy function.
\end{lemma}
The identity
\begin{equation*}
	\binom{n}{m}\binom{m}{\ell}=\binom{n}{\ell}\binom{n-\ell}{m-\ell}
\end{equation*}
will also be used in the context.

\section{Giant even induced subgraphs of undirected graphs}\label{sect_linearly.sized.gallai.even.induced.subgraphs}
Let $G$ be any undirected simple graph on $n$ vertices. We study the even induced subgraphs of order within $[\theta(n)/2,\theta(n)]$ in $G$ for all integer-valued functions $\theta(n)$ that satisfies $1\leq\theta(n)\leq n$, particularly the giant even induced subgraphs of $G$, namely considering the case when $\theta(n)=\alpha n$ for some constant $0<\alpha<1$.\par 
Recall that for every $W\subseteq V(G)$ with cardinality $\theta(n)$, i.e. $W\in\binom{V(G)}{\theta(n)}$, Theorem \ref{thm_gallai.even.partition} provides a primary co-even induced subgraph of $G[W]$, which has order within $[\theta(n)/2,\theta(n)]$. While it might be the case that different subsets of $V(G)$ produce common even induced subgraphs, and the even-even partition of a given subset $W\in\binom{V(G)}{\theta(n)}$ may not be unique. Therefore, we consider the number of co-even induced subgraphs of all $G[W]$, where $W\in\binom{V(G)}{\theta(n)}$, and provide a lower bound for the number of all even induced subgraphs $H$ of $G$ satisfying $\theta(n)/2\leq V(H)\leq\theta(n)$ by the following result.
\begin{theorem}\label{thm_number.of.arbitray.size.even.induced.subgraph.lower.bound}
	Let $\theta(n)$ be an integer-valued function of $n$ such that $1\leq\theta(n)\leq n$. Then the number of even induced subgraphs with order lying in the interval $[\theta(n)/2,\theta(n)]$ of any undirected simple graph $G$ on $n$ vertices is at least
	\begin{equation*}
		\binom{n}{n-\theta(n)}\Big\slash\binom{n-\frac{1}{2}\theta(n)}{n-\theta(n)}.
	\end{equation*}
	\begin{proof}
		Consider the primary co-even induced subgraphs of $G[W]$ for all sets $W\in\binom{V(G)}{\theta(n)}$, namely considering the set
		\begin{equation*}
			\Theta:=\left\{W'\subseteq W\mid W\in\binom{V(G)}{\theta(n)}~~\text{and $G[W']$ is a primary co-even induced subgraph of $G[W]$}\right\}.
		\end{equation*}\par 
		Let us define $\mathcal{E}_{\theta}$ to be the collection of the subset families $E_{\theta}$ of $V(G)$ where each $E_{\theta}$ in $\mathcal{E}_{\theta}$ satisfies the following two properties (P$_{1}$) and (P$_{2}$):\par 
		(P$_{1}$) Every set in $E_{\theta}$ has size at least $\theta(n)/2$;\par 
		(P$_{2}$) For every $W\in\binom{V(G)}{\theta(n)}$, there exists at least one $U\in E_{\theta}$ such that $U\subseteq W$.\par 
		\noindent Then it is straightforward to see that $\Theta\in\mathcal{E}_{\theta}$, which yields
		\begin{equation}\label{ineq_minimum.release.constraints}
			\min\left\{|E_{\theta}|\mid E_{\theta}\in\mathcal{E}_{\theta}\right\}\leq|\Theta|.
		\end{equation}
		For each $E_{\theta}\in\mathcal{E}_{\theta}$, let
		\begin{equation*}
			\widetilde{E_{\theta}}=\left\{Y\subseteq V(G)\mid V(G)\backslash Y\in E_{\theta}\right\}.
		\end{equation*}
		Then it is straightforward to see that $|E_{\theta}|=|\widetilde{E_{\theta}}|$, and every $\widetilde{E_{\theta}}$ satisfies the following complementary conditions of (P$_{1}$) and (P$_{2}$):\par 
		(P$_{1}'$) Every set in $\widetilde{E_{\theta}}$ has size less than $n-\theta(n)/2$;\par 
		(P$_{2}'$) For every $W\in\binom{V(G)}{n-\theta(n)}$, there exists $Y\in\widetilde{E_{\theta}}$ such that $W\subseteq Y$.\par 
		Since for each $Y\in\widetilde{E_{\theta}}$, there exist at most $\binom{n-\theta(n)/2}{n-\theta(n)}$ subsets $W\in\binom{V(G)}{n-\theta(n)}$ that are contained in $Y$, we have
		\begin{equation}\label{ineq_lower.bound.greedy.counting}
			|E_{\theta}|=|\widetilde{E_{\theta}}|\geq\frac{\binom{n}{n-\theta(n)}}{\binom{n-\frac{1}{2}\theta(n)}{n-\theta(n)}}
		\end{equation}
		for every $E_{\theta}\in\mathcal{E}_{\theta}$. Combining \eqref{ineq_minimum.release.constraints} and \eqref{ineq_lower.bound.greedy.counting} yields the result.
	\end{proof}
\end{theorem}
We mention that one cannot extend this proof of Theorem \ref{thm_number.of.arbitray.size.even.induced.subgraph.lower.bound} to the digraphs since even-even partition does not always exist in a digraph.\par 
Observe that Theorem \ref{thm_number.of.linear.size.even.induced.subgraph.lower.bound} follows directly from Theorem \ref{thm_number.of.arbitray.size.even.induced.subgraph.lower.bound}, since by Lemma \ref{lemma_binomial.coefficients.linear.size} we have
\begin{equation*}
	\begin{split}
		\binom{n}{(1-\alpha)n}\Big\slash\binom{(1-\frac{1}{2}\alpha)n}{(1-\alpha)n}=&2^{\left(1+o(1)\right)\left(H(1-\alpha)n-H(\frac{1-\alpha}{1-\frac{1}{2}\alpha})(1-\frac{1}{2}\alpha)n\right)}\\
		=&2^{\left(1+o(1)\right)\left(H(\frac{\alpha}{2})-\alpha\right)n}.
	\end{split}
\end{equation*}
Moreover, it can be easily checked that
\begin{equation*}
	H\left(\frac{\alpha}{2}\right)-\alpha>0\quad\text{for all $0<\alpha<1$},
\end{equation*}
from which we have the following examples, where the floor and ceiling signs are omitted for clarity.
\begin{example}
	The number of even induced subgraphs $H$ satisfying that $H$ is the primary co-even induced subgraph of some $G[W]$ with $|W|=0.99n$ is at least $2^{\left(1+o(1)\right)0.00993n}$.
\end{example}
The aim of presenting this example is to show that, even if one cannot use Theorem \ref{thm_gallai.even.partition} to produce even induced subgraph of any order by deleting vertices of $G$ iteratively (since this procedure terminates once the remaining subgraph $G[W]$ is already even), we still have an exponential lower bound for the number of even induced subgraphs that one can obtain from $G[W]$ by deleting a very small proportion of vertices $V(G)\backslash W$.\par 
\begin{example}
	Setting $\alpha=2/3$ we obtain at least $2^{\left(1+o(1)\right)0.25163n}$ giant induced subgraphs of order lying in the interval $[n/3,2n/3]$ around $n/2$.
\end{example}
For integrity, we also note that by taking $\theta(n)=n^{\beta}$ with constant $0<\beta<1$, one obtains that the number of even induced subgraphs $H$ of $G$ satisfying $n^{\beta}/2\leq V(H)\leq n^{\beta}$, if $1/2\leq\beta<1$, is at least
\begin{equation}\label{eq_number.of.even.induced.subgraphs.larger.than.square.lower.bound}
	\begin{split}
		\binom{n}{n-n^{\beta}}\Big\slash\binom{n-\frac{1}{2}n^{\beta}}{n-n^{\beta}}=&\binom{n}{n^{\beta}}\Big\slash\binom{n-\frac{1}{2}n^{\beta}}{\frac{1}{2}n^{\beta}}\\
		=&\exp\left\{\left(1+o(1)\right)\frac{1-\beta}{2}n^{\beta}\log{n}\right\}
	\end{split}
\end{equation}
following from Theorem \ref{thm_number.of.arbitray.size.even.induced.subgraph.lower.bound} and Lemma \ref{lemma_binomial.coefficients.larger.than.sqrt.size}. If $0<\beta<1/2$, then the number of even induced subgraphs $H$ satisfying $n^{\beta}/2\leq V(H)\leq n^{\beta}$ is at least
\begin{equation}\label{eq_number.of.even.induced.subgraphs.less.than.square.lower.bound}
	\binom{n}{n^{\beta}}\Big\slash\binom{n-\frac{1}{2}n^{\beta}}{\frac{1}{2}n^{\beta}}=\left(1+o(1)\right)\frac{\sqrt{2}}{2}\left(\frac{\mathrm{e}n^{1-\beta}}{2}\right)^{\frac{1}{2}n^{\beta}},
\end{equation}
following likewise from Theorem \ref{thm_number.of.arbitray.size.even.induced.subgraph.lower.bound} and Lemma \ref{lemma_binomial.coefficients.sqrt.size}.\par 
However, the lower bounds given by \eqref{eq_number.of.even.induced.subgraphs.larger.than.square.lower.bound} and \eqref{eq_number.of.even.induced.subgraphs.less.than.square.lower.bound} do not seem to be sharp as $n\to\infty$, given that the expected number of even induced subgraphs of order $r=n^{\beta}$ in random graph $G\in G_{n,p}$ is roughly $\binom{n}{r}2^{-r+1}$ by Proposition \ref{prop_expectation.of.number.of.even.induced.subgraphs.undirected}, where
\begin{equation*}
	\binom{n}{r}2^{-r+1}=2\exp\left\{\left(1+o(1)\right)(1-\beta)n^{\beta}\log{n}-n^{\beta}\log{2}\right\}\gtrsim\exp\left\{\left(1+o(1)\right)\frac{1-\beta}{2}n^{\beta}\log{n}\right\}
\end{equation*}
for all $1/2\leq\beta<1$ and
\begin{equation*}
	\binom{n}{r}2^{-r+1}\sim\sqrt{\frac{2}{\pi n^{\beta}}}\left(\frac{\mathrm{e}n^{1-\beta}}{2}\right)^{n^{\beta}}\gtrsim\frac{\sqrt{2}}{2}\left(\frac{\mathrm{e}n^{1-\beta}}{2}\right)^{\frac{1}{2}n^{\beta}}
\end{equation*}
for all $0<\beta<1/2$.

\section{Co-even induced subgraphs of Paley graphs}\label{sect_gallai.even.induced.subgraphs.paley}
\subsection{Characterization of co-even induced subgraphs}\label{subsect_characterization.co-even.induced.subgraphs}
As a complement to Theorem \ref{thm_caro.odd.parity.cover.and.gallai.partition}, for each graph $G$ we show the one-to-one correspondence between co-even induced subgraphs of $G$ and odd-parity covers of $\overline{G}$.
\begin{proposition}\label{prop_the.number.of.gallai.partition.equals.odd.parity.cover.in.odd.extension}
	Let $G$ be a graph. Then the number of co-even induced subgraphs of $G$ equals the number of odd-parity covers of the odd-extension $\overline{G}$.
	\begin{proof}
		Let $\Gamma_{e}$ be the set of the vertex sets of all co-even induced subgraphs of $G$ and let $\Gamma_{o}$ be the set of all odd-parity covers of $\overline{G}$. Then by Theorem \ref{thm_caro.odd.parity.cover.and.gallai.partition} we have a well-defined mapping
		\begin{equation*}
			\begin{split}
				\varphi:\Gamma_{o}&\to\Gamma_{e}\\
				Q&\mapsto Q\cap V(G).
			\end{split}
		\end{equation*}
		We define the mapping $\phi:\Gamma_{e}\to\Gamma_{o}$ by setting
		\begin{equation*}\label{eq_co.even.induced.subgraph.extend.to.odd.parity.cover}
			\phi(V_{1})=V_{1}\cup\left\{v'\mid v\in V(G)\backslash V_{1}~\text{and}~\deg_{G}(v)\equiv 0\Mod{2}\right\}
		\end{equation*}
		for each $V_{1}\in\Gamma_{e}$, where $v'$ is the extra vertex adjacent only to $v$. then it is straightforward to see that $\varphi$ is onto since we always have $V_{1}=\varphi(\phi(V_{1}))$. Therefore $|\Gamma_{e}|\leq|\Gamma_{o}|$.\par 
		For each $Q\in\Gamma_{o}$, we claim that for each $v'\in V(\overline{G})\backslash V(G)$,
		\begin{equation*}
			v'\in Q\iff v\notin Q\cap V(G)~~\text{and}~~\deg_{G}(v)\equiv 0\Mod{2}.
		\end{equation*}
		Actually, if $v'\in Q$, then we have $v\notin Q$ since otherwise $|N_{G}[v']\cap Q|=|\left\{v,v'\right\}|=2$, which contradicts to that $Q$ is an odd-parity cover. Conversely, assume that $v\notin Q$, $\deg_{G}(v)\equiv 0\Mod{2}$ and $v'\notin Q$, then we have $|N_{G}[v']\cap Q|=0$, which is also a contradiction. Thus, if we take $V_{1}=Q\cap V(G)=\varphi(Q)$, then $Q=\phi(V_{1})$. Therefore $\phi$ is onto and $|\Gamma_{o}|\leq|\Gamma_{e}|$. This completes the proof.
	\end{proof}
\end{proposition}
\begin{remark}
	In fact, both Theorem \ref{thm_caro.odd.parity.cover.and.gallai.partition} and Proposition \ref{prop_the.number.of.gallai.partition.equals.odd.parity.cover.in.odd.extension} can be generalized to all digraphs, while we do not present these results since Lemma \ref{lm_non.existence.co.even.induced.subgraphs.paley.directed} indicates that the Paley tournament $PT_{q}$ does not have an odd-parity cover.
\end{remark}
\subsection{Paley graphs}
We determine the number of co-even induced subgraphs of Paley graphs by proving Theorem \ref{thm_number.of.all.co.even.induced.subgraphs.paley}. The $2$-rank and the Smith group are needed for characterizing the null space of linear equation $(\overline{A}+I)X=\boldsymbol{1}$ over $\mathbb{F}_{2}$, where $\overline{A}$ stands for the adjacency matrix of the odd-extension of $P_{q}$.\par 
For a non-zero square integral matrix $M$ of order $m$, its \emph{Smith normal form} $S(M)$ is a diagonal matrix $S(M)=PMQ=\text{diag}(s_{1},...,s_{m})$, where $P$ and $Q$ are integral with determinant $\pm 1$ and
$s_{1}\mid s_{2}\mid\cdots\mid s_{m}$. The integers $s_{i}$ are called the \emph{elementary divisors} of $M$. The Smith normal form $S(M)$ exists and is uniquely determined up to the signs of the $s_{i}$. Since the integral matrix $M$ represents a $\mathbb{Z}$-homomorphism, we can define the \emph{Smith group} of $M$ to be its cokernel $\mathbb{Z}^{m}/M\mathbb{Z}^{m}$, the computation of which is equivalent to finding $S(M)$. By the fundamental theorem of finitely generated abelian groups, we have group isomorphisms
\begin{equation*}
	\mathbb{Z}^{m}/M\mathbb{Z}^{m}\cong\mathbb{Z}^{m}/S(M)\mathbb{Z}^{m}\cong\left(\mathbb{Z}/s_{1}\mathbb{Z}\right)\oplus\cdots\oplus\left(\mathbb{Z}/s_{\ell}\mathbb{Z}\right)\oplus\mathbb{Z}^{m-\ell}
\end{equation*}
for some $1\leq\ell\leq m$, and it follows that $S(M)=\text{diag}\left\{s_{1},s_{2},...,s_{\ell},\boldsymbol{0}\right\}$.\par 
For any prime $p$, the \emph{$p$-rank} of the matrix $M$, denoted by $\text{rank}_{p}(M)$, is defined to be the rank of $M$ over the finite field of order $p$. It is straightforward to see that $\text{rank}_{p}(M)$ is exactly the number of the elementary divisors $s_{i}$ that are not divisible by $p$. We refer to Chapter 13 of the book \cite{brouwer2012spectra} for more details.
\begin{proof}[Proof of Theorem \ref{thm_number.of.all.co.even.induced.subgraphs.paley}]
	Let $A$ be the adjacency matrix of the Paley graph $P_{q}$. It is shown in \cite[Theorem 2.1]{chandler2015smith} that the Smith group of $A$ is isomorphic to
	\begin{equation*}
		\mathbb{Z}/2\mu\mathbb{Z}\oplus\left(\mathbb{Z}/\mu\mathbb{Z}\right)^{2\mu},
	\end{equation*}
	where $\mu=(q-1)/4$. This yields
	\begin{equation*}
		\text{rank}_{2}(A)=\begin{dcases}
			\frac{q-1}{2} & \text{if $q\equiv 1\Mod{8}$},\\
			q-1 & \text{if $q\equiv 5\Mod{8}$}.
		\end{dcases}
	\end{equation*}
	Since $P_{q}$ has the valency $(q-1)/2$, which is even, the adjacency matrix of the odd-extension $\overline{P_{q}}$ is
	\begin{equation*}
		\overline{A}=\begin{pmatrix}
			A & I_{q}\\
			I_{q} & \boldsymbol{0}_{q}\\
		\end{pmatrix}.
	\end{equation*}
	It follows that the $2$-rank of $\overline{A}+I$ equals
	\begin{equation*}
		\text{rank}_{2}(\overline{A}+I)=\text{rank}_{2}(A)+q=\begin{dcases}
			\frac{3q-1}{2} & \text{if $q\equiv 1\Mod{8}$},\\
			2q-1 & \text{if $q\equiv 5\Mod{8}$},
		\end{dcases}
	\end{equation*}
	and the characteristic vectors of the odd-parity covers of $\overline{P_{q}}$ are given by the solution space of the linear system $(\overline{A}+I)X=\boldsymbol{1}$, which has a special solution whose first $q$ coordinates are all ones and the last $q$ coordinates are all zeros. Thus, as is shown in Proposition \ref{prop_the.number.of.gallai.partition.equals.odd.parity.cover.in.odd.extension}, we have
	\begin{equation*}
		N=|\ker(\overline{A}+I)|=2^{2q-\text{rank}_{2}(\overline{A}+I)}=2^{q-\text{rank}_{2}(A)},
	\end{equation*}
	which completes the proof.
\end{proof}
\begin{remark}
	The result of Theorem \ref{thm_number.of.all.co.even.induced.subgraphs.paley} shows that, if $q\equiv 5\Mod{8}$ then there only exist trivial co-even induced subgraphs of $P_{q}$. On the other hand, if $q\equiv 1\Mod{8}$, then the number of co-even induced subgraphs of $P_{q}$ equals $2^{(q+1)/2}\approx 1.41^{q+1}$. Since $P_{q}$ is quasi-random, we would expect that the number of even induced subgraphs of $P_{q}$ to be that number in the random graphs $G_{q,1/2}$. As indicated by Proposition \ref{prop_expectation.of.number.of.even.induced.subgraphs.undirected}, this expected value equals $\sum_{r=0}^{q}\binom{q}{r}2^{-r+1}=2\cdot 1.5^{q}$. In this sense, it can be seen that the co-even induced subgraphs predominate among all even induced subgraphs in $P_{q}$ when $q\equiv 1\Mod{8}$.
\end{remark}

\section{The number of even induced subgraphs of order $o(q^{1/4})$ in Paley graphs and Paley tournaments}\label{sect_even.induced.subgraphs.paley.small}
Recall that the number of labeled occurrences of any fixed (di)graph $H$ of order $r$ in Paley graph (tournament) is given by the definition of quasi-randomness. If the order $r=r(q)$ is not constant, then Theorem \ref{thm_paley.graph.r.full} shows, in particular, the existence of even induced subgraphs of order as small as $o(\log{q})$. In this section, we give asymptotic estimates for the number of even induced sub(di)graphs of order $r=o(q^{1/4})$ in $P_{q}$ and $PT_{q}$, respectively.
\subsection{Paley graphs}\label{subsect_small.even.induced.subgraphs.paley.graphs}
Before starting the proof of Theorem \ref{thm_number.of.even.induced.subgraphs.paley.undirected}, we state a few lemmas that are needed. Recall that a complex-valued function $f$ on $\Omega$, where $\Omega$ is an open set in $\mathbb{C}$, is said to be \emph{holomorphic} on $\Omega$ if for every $z\in\Omega$ the quotient $\left(f(z+h)-f(z)\right)/h$ converges to a limit when $h\to 0$. It is obvious that every polynomial is holomorphic on $\mathbb{C}$. The following is a classical result in complex analysis, see, e.g., \cite[Corollary 4.3]{stein2003complex}, which is also known as the \emph{saddle point bound} in the combinatorial context \cite[Chapter VIII]{flajolet2009analytic}.
\begin{lemma}\label{lemma_cauchy.inequalities}
	Let $f$ be a holomorphic function on an open set in $\mathbb{C}$ that contains the closure of a disc $D$ centered at $z_{0}$ and of radius $R$. Then
	\begin{equation*}
		|f^{(n)}(z_{0})|\leq\frac{n!\sup_{z\in\partial D}|f(z)|}{R^{n}}.
	\end{equation*}
\end{lemma}
Another ingredient in our proof is the character sum with polynomial arguments, also known as the \emph{Weil sum} over the finite field. In particular, we need the following version of Weil's theorem for estimating the value of the character sums.
\begin{lemma}[{\cite[Theorem 5.41]{lidl1996finite}}]
	Let $\psi$ be a non-trivial multiplicative character of $\mathbb{F}_{q}$ and let $m>1$ be its order. Let $f\in\mathbb{F}_{q}[x]$ be a monic polynomial of positive degree that is not an $m$-th power of a polynomial. Let $d$ be the number of distinct roots of $f$ in its splitting field over $\mathbb{F}_{q}$. Then
	\begin{equation}\label{eq_weil.bound}
		\left|\sum_{c\in\mathbb{F}_{q}}\psi\left(f(c)\right)\right|\leq(d-1)\sqrt{q}.
	\end{equation}
\end{lemma}
\begin{lemma}\label{lemma_saddle.point.estimate}
	Let $0\leq k\leq r\leq q$, where $r=\omega(1)$ and $r=o(\sqrt{q})$. For each $W\subseteq\mathbb{F}_{q}$ with cardinality $|W|=r-k$, let $f_{W}(x)=\prod_{w\in W}(x-w)\in\mathbb{F}_{q}[x]$ and define $f_{W}(x)=1$ if $W=\emptyset$. Let
	\begin{equation*}
		A_{k}(W)=\sum_{U\subseteq\mathbb{F}_{q}\backslash W\atop|U|=k}\prod_{u\in U}\eta\left(f_{W}(u)\right),
	\end{equation*}
	where $\eta$ is the quadratic character defined by $\eta(x)=x^{(q-1)/2}$ for all $x\in\mathbb{F}_{q}$, and the subscript $k$ denotes $r-|W|$. Then $A_{r}(\emptyset)=\binom{q}{r}$ and $A_{0}(W)=1$ whenever $|W|=r$. If $1\leq k\leq r-1$, then there exists a constant $C>0$ such that
	\begin{equation*}
		|A_{k}(W)|\leq\binom{q}{k}\left(\frac{C(r-k+\sqrt{k})}{\sqrt{q}}\right)^{k}.
	\end{equation*}
	\begin{proof}
		For integers $k$ satisfying $1\leq k\leq r-1$, consider the generating function
		\begin{equation*}
			F(\texttt{t}):=\prod_{u\in\mathbb{F}_{q}\backslash W}\left(1+\eta\left(f_{W}(u)\right)\texttt{t}\right),
		\end{equation*}
		where $W\subseteq\mathbb{F}_{q}$ and $|W|=r-k$, then $A_{k}(W)$ is exactly the coefficient of $\texttt{t}^{k}$ in $F(\texttt{t})$. For $t\in\mathbb{C}^*$ with $|t|<1$, we compute
		\begin{equation}\label{eq_log.modulus.of.F(t)}
			\begin{split}
				\log{|F(t)|}=&\sum_{u\in\mathbb{F}_{q}\backslash W}\log{\left|1+\eta\left(f_{W}(u)\right)t\right|}\\
				=&\Re{\left(\sum_{u\in\mathbb{F}_{q}\backslash W}\log{\left(1+\eta\left(f_{W}(u)\right)t\right)}\right)}\\
				=&\Re{\left(\sum_{u\in\mathbb{F}_{q}\backslash W}\sum_{n=1}^{\infty}(-1)^{n+1}\frac{\eta\left(f_{W}(u)\right)^{n}}{n}t^{n}\right)}\\
				=&\Re{\left(\sum_{n=1}^{\infty}\frac{(-1)^{n+1}}{n}\left(\sum_{u\in\mathbb{F}_{q}\backslash W}\eta\left(f_{W}(u)^{n}\right)\right)t^{n}\right)}.
			\end{split}
		\end{equation}\par 
		If $n$ is even, then we simply have $\sum_{u\in\mathbb{F}_{q}\backslash W}\eta\left(f_{W}(u)^{n}\right)=q-(r-k)$. This implies
		\begin{equation}\label{eq_re.even.terms.character.sum}
			\begin{split}
				\Re{\left(\sum_{n\geq 1\atop\text{$n$ is even}}\frac{(-1)^{n+1}}{n}\left(\sum_{u\in\mathbb{F}_{q}\backslash W}\eta\left(f_{W}(u)^{n}\right)\right)t^{n}\right)}=&\frac{q-(r-k)}{2}\Re{\left(\log{(1-t^{2})}\right)}\\
				=&\frac{q-(r-k)}{2}\log{|1-t^{2}|}\\
				\leq&\frac{q-(r-k)}{2}|t|^{2}.
			\end{split}
		\end{equation}
		If $n$ is odd, then the Weil's bound \eqref{eq_weil.bound} yields
		\begin{equation*}
			\left|\sum_{u\in\mathbb{F}_{q}\backslash W}\eta\left(f_{W}(u)^{n}\right)\right|=\left|\sum_{u\in\mathbb{F}_{q}}\eta\left(f_{W}(u)^{n}\right)\right|\leq(r-k-1)\sqrt{q}<(r-k)\sqrt{q}
		\end{equation*}
		which implies
		\begin{equation}\label{ineq_odd.terms.character.sum}
			\begin{split}
				\Re{\left(\sum_{n\geq 1\atop\text{$n$ is odd}}\frac{(-1)^{n+1}}{n}\left(\sum_{u\in\mathbb{F}_{q}\backslash W}\eta\left(f_{W}(u)^{n}\right)\right)t^{n}\right)}\leq&\sum_{n\geq 1\atop\text{$n$ is odd}}\left|\sum_{u\in\mathbb{F}_{q}\backslash W}\eta\left(f_{W}(u)^{n}\right)\right|\frac{(-1)^{n+1}}{n}|t|^{n}\\
				\leq&(r-k)\sqrt{q}\sum_{n\geq 1\atop\text{$n$ is odd}}\frac{|t|^{n}}{n}\\
				=&\frac{(r-k)\sqrt{q}}{2}\log{\left(1+\frac{2|t|}{1-|t|}\right)}\\
				\leq&(r-k)\sqrt{q}\frac{|t|}{1-|t|}.
			\end{split}
		\end{equation}
		Given that $|t|<1$, plugging \eqref{eq_re.even.terms.character.sum} and \eqref{ineq_odd.terms.character.sum} into \eqref{eq_log.modulus.of.F(t)} yields
		\begin{equation*}
			\log{|F(t)|}\leq(r-k)\sqrt{q}\frac{|t|}{1-|t|}+\frac{1}{2}(q-(r-k))|t|^{2}.
		\end{equation*}	
		Thus, by Lemma \ref{lemma_cauchy.inequalities}
		\begin{equation}\label{ineq_Ak(W).Cauchy.inequality}
			\begin{split}
				|A_{k}(W)|=\left|\frac{F^{(k)}(0)}{k!}\right|\leq&\frac{\sup_{|t|=\rho}|F(t)|}{\rho^{k}}\\
				\leq&\exp{\left\{\left(r-k\right)\sqrt{q}\frac{\rho}{1-\rho}+\frac{1}{2}\left(q-(r-k)\right)\rho^{2}-k\log{\rho}\right\}}
			\end{split}
		\end{equation}
		for all $0<\rho<1$. Let $0<\varepsilon<1$ be a constant. Then for all $0<\rho<\varepsilon$, there exists $\delta>0$ such that $1/(1-\rho)<\delta$. If we consider the function
		\begin{equation*}
			g(\rho):=A\sqrt{q}\rho+\frac{1}{2}B\rho^{2}-k\log{\rho},
		\end{equation*}
		where $A=\delta(r-k)$ and $B=q-(r-k)$, then \eqref{ineq_Ak(W).Cauchy.inequality} yields
		\begin{equation}\label{ineq_Ak(W).upper.bound.g(rho)}
			|A_{k}(W)|\leq\exp{g(\rho)}
		\end{equation}
		for all $0<\rho<\varepsilon$.\par 
		We establish the upper bound for $|A_{k}(W)|$ through optimizing the value of $g(\rho)$. Let $\rho^*$ denote the positive solution to
		\begin{equation*}
			\rho\frac{\text{d}}{\text{d}\rho}g(\rho)=\rho\frac{\text{d}}{\text{d}\rho}\left(A\sqrt{q}\rho+\frac{1}{2}B\rho^{2}-k\log{\rho}\right)=A\sqrt{q}\rho+B\rho^{2}-k=0,
		\end{equation*}
		in other words
		\begin{equation}\label{eq_minimum.point.expression}
			\rho^*=\frac{\sqrt{A^{2}q+4kB}-A\sqrt{q}}{2B}=\frac{2k}{\sqrt{A^{2}q+4kB}+A\sqrt{q}}.
		\end{equation}
		Then it is straightforward to see that $\rho^*$ is the minimum point of $g(\rho)$ on $(0,\varepsilon)$ if $q$ is sufficiently large, given that $r=o(\sqrt{q})$. This implies
		\begin{equation*}
			\begin{split}
				\min_{\rho\in(0,\varepsilon)}g(\rho)=g(\rho^*)=&A\sqrt{q}\rho^*+\frac{1}{2}B{\rho^*}^{2}-k\log{\rho^*}\\
				=&-\frac{1}{2}B{\rho^*}^{2}+k-k\log{\rho^*}\\
				\leq&k\log{\frac{q}{k}}+k+k\log{\frac{k}{\rho^*q}}.
			\end{split}
		\end{equation*}
		Note that by \eqref{eq_minimum.point.expression} we have
		\begin{equation*}
			\frac{k}{\rho^*q}=\frac{\sqrt{A^{2}q+4kB}+A\sqrt{q}}{2q}\leq\frac{\sqrt{A^{2}+4k}+A}{2\sqrt{q}}=\frac{2k}{\sqrt{q}\left(\sqrt{A^{2}+4k}-A\right)},
		\end{equation*}
		which means $k\log{(k/(\rho^*q))}\leq k\log{(k/\sqrt{q})}+k\log{2}-k\log{\left(\sqrt{A^{2}+4k}-A\right)}$. It is routine to check that
		\begin{equation*}
			\sqrt{A^{2}+4k}-A\geq\frac{2k}{A+\sqrt{k}}.
		\end{equation*}
		It follows that
		\begin{equation}\label{ineq_g(rho*).extraction}
			\begin{split}
				g(\rho^*)\leq& k\log{\frac{q}{k}}+k+k\log\frac{k}{\sqrt{q}}+k\log{2}-k\log{\left(\sqrt{A^{2}+4k}-A\right)}\\
				\leq&k\log{\frac{q}{k}}+k+k\log\frac{k}{\sqrt{q}}+k\log(\frac{A+\sqrt{k}}{k})\\
				=&k\log{\frac{q}{k}}+k\log\frac{\delta(r-k)+\sqrt{k}}{\sqrt{q}}+k.
			\end{split}
		\end{equation}
		Therefore, by \eqref{ineq_Ak(W).upper.bound.g(rho)} and \eqref{ineq_g(rho*).extraction}, there exists a positive constant $C=C(\delta)$ such that
		\begin{equation*}
			|A_{k}(W)|\leq\mathrm{e}^{k\log{\frac{q}{k}}}\left(\frac{C(r-k+\sqrt{k})}{\sqrt{q}}\right)^{k}\leq\binom{q}{k}\left(\frac{C(r-k+\sqrt{k})}{\sqrt{q}}\right)^{k}.
		\end{equation*}
		This completes the proof.
	\end{proof}
\end{lemma}
\begin{proof}[Proof of Theorem \ref{thm_number.of.even.induced.subgraphs.paley.undirected}]
	For every subset $S\subseteq\mathbb{F}_{q}$ with cardinality $r$, we define the indicator variable
	\begin{equation*}
		\mathds{1}_{S}=\prod_{v\in S}\frac{1+(-1)^{\deg_{P_{q}[S]}(v)}}{2},
	\end{equation*}
	then $S$ induces an even subgraph if $\mathds{1}_{S}=1$, otherwise $\mathds{1}_{S}=0$. Therefore
	\begin{equation*}
		\begin{split}
			N_{r}=&\sum_{S\subseteq\mathbb{F}_{q}\atop |S|=r}\mathds{1}_{S}=\sum_{S\subseteq\mathbb{F}_{q}\atop |S|=r}2^{-r}\sum_{T\subseteq S}\prod_{v\in T}(-1)^{\deg_{P_{q}[S]}(v)}\\
			=&\sum_{S\subseteq\mathbb{F}_{q}\atop |S|=r}2^{-r}\sum_{T\subseteq S}(-1)^{2e(T)+e(T,S\backslash T)}=\sum_{S\subseteq\mathbb{F}_{q}\atop |S|=r}2^{-r}\sum_{T\subseteq S}(-1)^{e(T,S\backslash T)},
		\end{split}
	\end{equation*}
	where we denote by $e(U,W)$ the cardinality of the set of edges $\left\{(u,w)\mid u\in U,w\in W,(u,w)\in E(P_{q})\right\}$ and write $e(U)=e(U,U)$. Observe that for all $T\subseteq S\subseteq\mathbb{F}_{q}$, we have
	\begin{equation*}
		(-1)^{e(T,S\backslash T)}=(-1)^{|T|(|S|-|T|)}\prod_{u\in T}\prod_{w\in S\backslash T}\eta(u-w),
	\end{equation*}
	which yields
	\begin{equation*}
		\begin{split}
			N_{r}=&\sum_{S\subseteq\mathbb{F}_{q}\atop |S|=r}2^{-r}\sum_{T\subseteq S}(-1)^{|T|(|S|-|T|)}\prod_{u\in T}\prod_{w\in S\backslash T}\eta(u-w)\\
			=&\sum_{S\subseteq\mathbb{F}_{q}\atop |S|=r}2^{-r}\sum_{k=0}^{r}(-1)^{k(r-k)}\sum_{T\subseteq S\atop |T|=k}\prod_{u\in T}\eta\left(\prod_{w\in S\backslash T}(u-w)\right)\\
			=&2^{-r}\sum_{k=0}^{r}(-1)^{k(r-k)}\sum_{U,W\subseteq\mathbb{F}_{q}, U\cap W=\emptyset\atop |U|=k, |W|=r-k}\prod_{u\in U}\eta\left(\prod_{w\in W}(u-w)\right)\\
			=&2^{-r}\sum_{k=0}^{r}(-1)^{k(r-k)}\sum_{W\subseteq\mathbb{F}_{q}\atop|W|=r-k}\sum_{U\subseteq\mathbb{F}_{q}\backslash W\atop|U|=k}\prod_{u\in U}\eta\left(\prod_{w\in W}(u-w)\right).
		\end{split}
	\end{equation*}
	Therefore,
	\begin{equation*}
		\begin{split}
			N_{r}=&2^{-r}\sum_{k=0}^{r}(-1)^{k(r-k)}\sum_{W\subseteq\mathbb{F}_{q}\atop|W|=r-k}A_{k}(W)\\
			=&2^{1-r}\binom{q}{r}+2^{-r}\sum_{k=1}^{r-1}(-1)^{k(r-k)}\sum_{W\subseteq\mathbb{F}_{q}\atop|W|=r-k}A_{k}(W).
		\end{split}
	\end{equation*}
	Hence, by Lemma \ref{lemma_saddle.point.estimate}, we have
	\begin{equation*}
		\begin{split}
			\left|N_{r}-2^{1-r}\binom{q}{r}\right|=&2^{-r}\left|\sum_{k=1}^{r-1}(-1)^{k(r-k)}\sum_{W\subseteq S\atop|W|=r-k}A_{k}(W)\right|\\
			\leq&2^{-r}\sum_{k=1}^{r-1}\binom{q}{r-k}|A_{k}(W)|\\
			\leq&2^{-r}\sum_{k=1}^{r-1}\binom{q}{r-k}\binom{q}{k}\left(\frac{C(r-k+\sqrt{k})}{\sqrt{q}}\right)^{k}.
		\end{split}
	\end{equation*}
	Notice that since $r=o(\sqrt{q})$, by Lemma \ref{lemma_binomial.coefficients.sqrt.size} we have $\binom{q}{r-k}\binom{q}{k}/\binom{q}{r}=\binom{r}{k}\binom{q}{r-k}/\binom{q-k}{r-k}=\left(1+o(1)\right)\binom{r}{k}$. This yields
	\begin{equation*}
		\sum_{k=1}^{r-1}\frac{\binom{q}{r-k}\binom{q}{k}}{\binom{q}{r}}\left(\frac{C(r-k+\sqrt{k})}{\sqrt{q}}\right)^{k}=\left(1+o(1)\right)\sum_{k=1}^{r-1}\binom{r}{k}\left(\frac{C(r-k+\sqrt{k})}{\sqrt{q}}\right)^{k}.
	\end{equation*}
	Let $\zeta=Cr/\sqrt{q}$. Then we have $\zeta\geq C(r-k+\sqrt{k})/\sqrt{q}$ and $\zeta=o(1)$. It follows that
	\begin{equation*}
		\sum_{k=1}^{r-1}\binom{r}{k}\zeta^{r}=(1+\zeta)^{r}-1-\zeta^{r},
	\end{equation*}
	where $(1+\zeta)^{r}=\exp{\left\{r\log(1+\zeta)\right\}}=\exp{\left\{r(\zeta+O(\zeta^{2}))\right\}}=1+O(r\zeta)$, given that $r\zeta=O(r^{2}/\sqrt{q})=o(1)$. Hence, we have
	\begin{equation*}
		\sum_{k=1}^{r-1}\binom{q}{r-k}\binom{q}{k}\zeta^{k}=\binom{q}{r}O(r\zeta)=\binom{q}{r}O\left(\frac{r^{2}}{\sqrt{q}}\right).
	\end{equation*}
	This finishes the proof.
\end{proof}
\subsection{Paley tournaments}
\begin{proof}[Proof of Theorem \ref{thm_number.of.even.induced.subgraphs.paley.directed}]
	Analogously to the case of the Paley graphs, for every subset $S\subseteq\mathbb{F}_{q}$ with cardinality $r$, we define the indicator variable
	\begin{equation*}
		\tilde{\mathds{1}}_{S}=\prod_{v\in S}\frac{1+(-1)^{\deg_{PT_{q}[S]}^{-}(v)}}{2},
	\end{equation*}
	then $S$ induces an even subdigraph of $PT_{q}$ if $\tilde{\mathds{1}}_{S}=1$, otherwise $\tilde{\mathds{1}}_{S}=0$. Therefore, if we denote by $e^{-}(U,W)$ the number of all directed edges from $U$ to $W$, we have
	\begin{equation}\label{eq_Nr.quadratic.characters.sum.directed}
		\begin{split}
			\tilde{N}_{r}=&\sum_{S\subseteq\mathbb{F}_{q}\atop|S|=r}\tilde{\mathds{1}}_{S}=\sum_{S\subseteq\mathbb{F}_{q}\atop|S|=r}2^{-r}\sum_{T\subseteq S}\prod_{v\in T}(-1)^{\deg_{PT_{q}[S]}^{-}(v)}\\
			=&\sum_{S\subseteq\mathbb{F}_{q}\atop|S|=r}2^{-r}\sum_{T\subseteq S}(-1)^{e^{-}(T,T)+e^{-}(T,S\backslash T)}\\
			=&2^{-r}\sum_{k=0}^{r}\sum_{U\subseteq\mathbb{F}_{q}\atop|U|=k}(-1)^{e^{-}(U,U)}\sum_{W\subseteq\mathbb{F}_{q}\backslash U\atop|W|=r-k}(-1)^{e^{-}(U,W)}\\
			=&2^{-r}\sum_{k=0}^{r}(-1)^{k(r-k)}\sum_{U\subseteq\mathbb{F}_{q}\atop|U|=k}(-1)^{e^{-}(U,U)}\sum_{W\subseteq\mathbb{F}_{q}\backslash U\atop|W|=r-k}\prod_{w\in W}\eta\left(\prod_{u\in U}(u-w)\right),
		\end{split}
	\end{equation}
	in which the last equation holds by the observation $(-1)^{e^{-}(U,W)}=(-1)^{|U||W|}\prod_{u\in U}\prod_{w\in W}\eta(u-w)$ for all $U\cap W=\emptyset$. Observe that for each $U\subseteq\mathbb{F}_{q}$ with $|U|=k$, we have
	\begin{equation*}
		(-1)^{e^{-}(U,U)}=(-1)^{k(k-1)}\prod_{u\in U}\prod_{u'\in U,u'\neq u}\eta(u-u')=\prod_{\left\{u,u'\right\}\subseteq U}\eta(u-u')\eta(u'-u)=(-1)^{\binom{k}{2}},
	\end{equation*}
	which yields for each $0\leq k\leq r$,
	\begin{equation*}
		\begin{split}
			\sum_{U\subseteq\mathbb{F}_{q}\atop|U|=k}(-1)^{e^{-}(U,U)}\sum_{W\subseteq\mathbb{F}_{q}\backslash U\atop|W|=r-k}\prod_{w\in W}\eta\left(\prod_{u\in U}(u-w)\right)=&(-1)^{\binom{k}{2}}\sum_{W\subseteq\mathbb{F}_{q}\atop|W|=r-k}\sum_{U\subseteq\mathbb{F}_{q}\backslash W\atop|U|=k}\prod_{u\in U}\eta\left(\prod_{w\in W}(u-w)\right)\\
			=&(-1)^{\binom{k}{2}}\sum_{W\subseteq\mathbb{F}_{q}\atop|W|=r-k}A_{k}(W).
		\end{split}
	\end{equation*}
	Therefore, noticing $(-1)^{\binom{r}{2}}=1$ since $r\equiv 0,1\Mod{4}$, it follows from \eqref{eq_Nr.quadratic.characters.sum.directed} that
	\begin{equation*}
		\left|\tilde{N}_{r}-2^{1-r}\binom{q}{r}\right|\leq2^{-r}\sum_{k=1}^{r-1}\sum_{W\subseteq\mathbb{F}_{q}\atop|W|=r-k}|A_{k}(W)|,
	\end{equation*}
	and the proof is complete.
\end{proof}
We remark that the result of Theorem \ref{thm_number.of.even.induced.subgraphs.paley.directed} accords with Proposition \ref{prop_expectation.of.number.of.even.induced.subgraphs.directed}, since the mean value of $\tilde{N}_{r}$ among all possible values of $r$ equals the expected number $2^{-r}\binom{q}{r}$ for the random digraphs.\par 
Both Theorem \ref{thm_number.of.even.induced.subgraphs.paley.undirected} and Theorem \ref{thm_number.of.even.induced.subgraphs.paley.directed} (for $r\equiv 0,3\Mod{4}$) hold for odd induced sub(di)graphs of order $o(q^{1/4})$ by similar arguments with only minor adaptations of the indicator variables.

\section{Applications to the construction of MDS self-dual codes}\label{sect_applications.to.mds.self.dual}
As is mentioned in $\S$\ref{subsect_motivation.from.coding.theory}, the construction of MDS self-dual codes via (extended) GRS codes is equivalent to enumerating induced subgraphs with the degree parity conditions in $P_{q}$ and $PT_{q}$. In the literature of MDS self-dual codes, there have been several existential statements that essentially coincide with the previous results in graph theory. For instance, it was proved in \cite[Lemma 13]{wan2023newmds} that for any odd prime power $q$, if 
\begin{equation}\label{ineq_paley.clique.full.bound}
	q>\left((r-3)2^{r-3}+\frac{1}{2}+\sqrt{\left((r-3)2^{r-3}+\frac{1}{2}\right)^{2}+(r-1)2^{r-2}}\right)^{2},
\end{equation}
then there exists a subset $\mathcal{S}=\left\{\alpha_{1},\alpha_{2},...,\alpha_{r}\right\}\subseteq\mathbb{F}_{q}$ such that $\eta(\alpha_{j}-\alpha_{i})=1$ for all $1\leq i<j\leq r$. Note that \eqref{ineq_paley.clique.full.bound} is exactly the same as the bound given by \eqref{ineq_paley.graph.r.full.bound}, and this is a special case of Theorem \ref{thm_paley.graph.r.full} when $q\equiv 1\Mod{4}$ since the subset $\mathcal{S}$ here induces a clique in $P_{q}$. This improves the result of \cite[Theorem 3.2(ii)]{jin2017new} proved for $q\geq r^{2}4^{r}$, which recovers the earlier bound given by Bollob{\'a}s.\par 
Meanwhile, a large number of MDS self-dual codes have been explicitly constructed from (extended) GRS codes. Write $q=p^{e}$ with odd prime $p$ and some positive integer $e$. Then we can classify some of the explicit constructions into Types I, II, III and IV in Table \ref{table_mds.self.dual.codes} based on their conditions on $q$.\par 
\begin{table}[!ht]\label{table_mds.self.dual.codes}
	\centering
	\begin{threeparttable}
		\caption{Known explicitly constructed MDS self-dual codes of length $n$ from (extended) GRS codes}
		\begin{tabular}{m{0.7cm}<{\centering}m{3.5cm}<{\centering}m{7cm}<{\centering}m{1.8cm}<{\centering}}
			\toprule
			\textbf{Type} & \textbf{Restrictions on $q$} & \textbf{Length $n$} & \textbf{References}\\
			\midrule
			I-1 & $q$ odd & $(n-2)\mid(q-1)$, $\eta(2-n)=1$ & \cite{yan2019note,fang2019new}\\
			\addlinespace
			I-2 & $q$ odd & $(n-1)\mid(q-1)$, $\eta(1-n)=1$ & \cite{yan2019note}\\
			\addlinespace
			I-3 & $q$ odd & $n=p^{\ell}+1$, $\ell\leq e-1$ & \cite{fang2019new,lebed2022some}\\
			II-1 & $q=p^{e}$, $e=uv$ & $n=2tp^{u\ell}$, $1\leq\ell\leq v-1$, $2t\mid(p^{u}-1)$, $2\mid\frac{q-1}{2t}$ & \cite{fang2020new}\\
			\addlinespace
			II-2 & $q=k^{s}$, $k$ odd, $s\geq 2$ & $n=tk^{z}+1$, $1\leq t<k-1$, $2\nmid t$, $t\mid(k-1)$, $1\leq z\leq s-1$, $\eta((-1)^{\frac{k^{z}+1}{2}}t)=1$ & \cite{lebed2022some}\\
			\addlinespace
			III-1 & $q\equiv 1\Mod{4}$ & $n\mid(q-1)$ & \cite{yan2019note}\\
			\addlinespace
			III-2 & $q\equiv 1\Mod{4}$ & $n=p^{\ell}+1$, $\ell\leq e-1$ & \cite{fang2019new}\\
			\addlinespace
			III-3 & $q\equiv 1\Mod{4}$ & $n=2p^{\ell}$, $\ell\leq e-1$ & \cite{yan2019note,fang2019new}\\
			\addlinespace
			IV-1 & $q=p^{e}$, $e=2s$ & $n=(2t+1)p^{s'\ell}+1$, $s'\mid s$, $0\leq\ell<\frac{e}{s'}$, $0\leq t\leq\frac{p^{s'}-1}{2}$ & \cite{fang2019new}\\
			\addlinespace
			IV-2 & $q=p^{e}$, $e=2s$ & $n=2tp^{s'\ell}$, $s'\mid s$, $0\leq\ell<\frac{e}{s'}$, $0\leq t\leq\frac{p^{s'}-1}{2}$ & \cite{fang2019new}\\
			\addlinespace
			IV-3 & $q=p^{e}$, $e=2s$ & $n=p^{2\ell}+1$, $1\leq\ell\leq s-1$ & \cite{fang2020new}\\
			\addlinespace
			IV-4 & $q=k^{2}$, $k\equiv 3\Mod{4}$ & $n=2tk$, $t\leq\frac{k-1}{2}$ & \cite{jin2017new}\\
			\addlinespace
			IV-5 & $q=k^{2}$, $k$ odd & $n\leq k$ & \cite{jin2017new}\\
			\addlinespace
			IV-6 & $q=k^{2}$, $k$ odd & (i) $n=tm+c$ with $m\mid(q-1)$ and extra conditions on $t$ and $c$,\newline
			(ii) $n=sa+tb+c$ with $a,b\mid(q-1)$, $c=0$, $1$, $2$ and extra conditions on $s,t$ and $c$,\newline
			(iii) $n=sa+tb+std+c$ and extra conditions on $s,t,a,b,c$ and $d$, etc.&  \cite{lebed2019construction,zhang2019new,zhang2022construction,fang2020new,fang2021new,fang2022new,fang2021construction,lebed2022some,huang2023generic,wan2023newmds,wan2023new,fang2026construction}\\
			\bottomrule
		\end{tabular}
		\begin{tablenotes}
			\item For Type IV-6, restrictions on $k\Mod{4}$ are actually imposed for different constructions of $n$. The exact conditions for the parameters in Type IV-6 are not provided completely due to the limit of space. For a more detailed list of known explicit constructions, we refer to the very recent work \cite{fang2026construction}.
		\end{tablenotes}
	\end{threeparttable}
\end{table}
These constructions admit wide ranges of possible length $n$, for instance, the ratio $\mathcal{N}_{q}/(q/2)$ could be more than 85\% for some specific values of the prime power $q$ satisfying $q=k^{2}$ and $k\equiv 3\Mod{4}$ in the construction of \cite{fang2026construction}, where $\mathcal{N}_{q}$ is the number of all the MDS self-dual codes that are constructed. Most of the results were obtained by considering the evaluation set consists of multiplicative subgroups of $\mathbb{F}_{q}^*$ and their cosets, or a subspace of $\mathbb{F}_{q}$ and its cosets, with rather sophisticated techniques.\par 
However, the restrictions imposed on $q$ certainly place limitations on the possible length $n$. The constructions of Type IV are subject to that $q$ is a square, the constructions of Type II, Type I-3, Type III-2 and III-3 contribute to the case when $q$ is not a prime. If we take a prime $q\geq 7$ such that $(q-1)/2$ is also prime (there are 229567 such primes $q$ within $10^{8}$), then $q\equiv 3\Mod{4}$ Type I-1, I-2 and Type III-1 admit only possibly $n=4$, $n=(q+1)/2$ and $n=2,q-1$, respectively. Likewise, if $q\geq 13$ and $(q-1)/4$ are both prime (there are 119718 such $q$ within $10^{8}$), then $q\equiv 5\Mod{8}$ and Type I-1, I-2 and Type III-1 admit only possibly $n=4,6,(q+3)/2$, $n=2,(q+3)/4$ and $n=2,4,(q-1)/2,q-1$, respectively.\par 
Therefore, we present our results in the language of coding theory as a complement to the results that are listed in Table \ref{table_mds.self.dual.codes}, showing the existence and the quantity of all the MDS self-dual codes that could be constructed via (extended) GRS codes.\par 
For every number $n$, if $n$ is even, let $\Omega(n,q)$ be the set of all subsets $\mathcal{S}\subseteq\mathbb{F}_{q}$ with cardinality $|\mathcal{S}|=n$ that satisfy item (i) in Theorem \ref{thm_grs.self.dual.unified.criteria}. If $n$ is odd, let $\tilde{\Omega}(n+1,q)$ be the set of all subsets $\mathcal{S}\subseteq\mathbb{F}_{q}$ with cardinality $|\mathcal{S}|=n$ that satisfy item (ii) in Theorem \ref{thm_grs.self.dual.unified.criteria}. Then $|\Omega(n,q)|$ is the number of all MDS self-dual codes of length $n$ that can be constructed via $[n,n/2,n/2+1]$-GRS codes and $|\tilde{\Omega}(n+1,q)|$ is the number of all MDS self-dual codes of length $n+1$ that can be constructed via $[n+1,(n+1)/2,(n+3)/2]$-extended GRS codes.\par 
We can therefore describe the correspondence between the MDS self-dual codes and even/odd induced sub(di)graphs of $P_{q}$ and $PT_{q}$ more precisely. If $q\equiv 1\Mod{4}$, then
\begin{equation}\label{eq_grs.degree.parity.even.order.paley}
	\begin{split}
		\Omega(n,q)=&\left\{\mathcal{S}\subseteq\mathbb{F}_{q}\mid |\mathcal{S}|=n~~\text{and $P_{q}[\mathcal{S}]$ has all degrees even}\right\}\\
		&\cup\left\{\mathcal{S}\subseteq\mathbb{F}_{q}\mid |\mathcal{S}|=n~~\text{and $P_{q}[\mathcal{S}]$ has all degrees odd}\right\},
	\end{split}
\end{equation}
and if $q\equiv 3\Mod{4}$, then
\begin{equation}\label{eq_grs.degree.parity.even.order.paley.tournament}
	\begin{split}
		\Omega(n,q)=&\left\{\mathcal{S}\subseteq\mathbb{F}_{q}\mid |\mathcal{S}|=n~~\text{and $PT_{q}[\mathcal{S}]$ has all out-degrees even}\right\}\\
		&\cup\left\{\mathcal{S}\subseteq\mathbb{F}_{q}\mid |\mathcal{S}|=n~~\text{and $PT_{q}[\mathcal{S}]$ has all out-degrees odd}\right\}.
	\end{split}
\end{equation}
Likewise, if $q\equiv 1\Mod{4}$, then
\begin{equation}\label{eq_egrs.even.induced.odd.order.paley}
	\tilde{\Omega}(n+1,q)=\left\{\mathcal{S}\subseteq\mathbb{F}_{q}\mid |\mathcal{S}|=n~~\text{and $P_{q}[\mathcal{S}]$ has all degrees even}\right\},
\end{equation}
if $q\equiv 3\Mod{4}$, then we have
\begin{equation}\label{eq_egrs.odd.induced.odd.order.paley.tournament}
	\tilde{\Omega}(n+1,q)=\left\{\mathcal{S}\subseteq\mathbb{F}_{q}\mid |\mathcal{S}|=n~~\text{and $PT_{q}[\mathcal{S}]$ has all out-degrees odd}\right\}.
\end{equation}
\begin{theorem}
	If $q\equiv 1\Mod{4}$, then for all constant $0<\alpha<1$,
	\begin{equation*}
		\left|\bigcup_{\alpha q/2<n<\alpha q}\Omega(n,q)\cup\tilde{\Omega}(n+1,q)\right|\geq 2^{\left(1+o(1)\right)\left(H(\frac{\alpha}{2})-\alpha\right)q}.
	\end{equation*}
	\begin{proof}
		This follows immediately from Theorem \ref{thm_number.of.linear.size.even.induced.subgraph.lower.bound} and Equations \eqref{eq_grs.degree.parity.even.order.paley} and \eqref{eq_egrs.even.induced.odd.order.paley}.
	\end{proof}
\end{theorem}
\begin{theorem}
	Let $n$ be an even number that satisfies $2\leq n\leq q-1$. Let $\mathcal{A}$ denote the collection of the subsets $\mathcal{S}=\left\{\alpha_{1},\alpha_{2},...,\alpha_{n}\right\}\subseteq\mathbb{F}_{q}$ where each $\mathcal{S}$ satisfies the following conditions:\par 
	(i) $\eta(\Delta_{\mathcal{S}}(\alpha_{i}))=1$ for all $i=1,2,...,n$;\par 
	(ii) Let $\left\{\beta_{1},\beta_{2},...,\beta_{q-n}\right\}$ be the complement of $\mathcal{S}$. Then $\eta(\Delta_{\mathbb{F}_{q}\backslash\mathcal{S}}(\beta_{s}))=1$ for all $s=1,2,...,q-n$.\par 
	If $q\equiv 1\Mod{8}$, then every $\mathcal{S}\in\mathcal{A}$ admits an MDS self-dual code of length $n$ and $\mathbb{F}_{q}\backslash\mathcal{S}$ admits an MDS self-dual code of length $q-n+1$, and we have
	\begin{equation*}
		|\mathcal{A}|=2^{\frac{q-1}{2}}-1.
	\end{equation*}
	Otherwise, we have $\mathcal{A}=\emptyset$.
	\begin{proof}
		This follows immediately from Theorem \ref{thm_number.of.all.co.even.induced.subgraphs.paley} and Lemma \ref{lm_non.existence.co.even.induced.subgraphs.paley.directed}.
	\end{proof}
\end{theorem}
\begin{theorem}
	If $q\equiv 3\Mod{4}$, then we have $\Omega(n,q)=\emptyset$ for all $n\equiv 2\Mod{4}$ and $\tilde{\Omega}(n+1,q)=\emptyset$ for all $n\equiv 1\Mod{4}$. If $q\equiv 3\Mod{4}$, $n=\omega(1)$ and $n=o(q^{1/4})$, then we have
	\begin{equation*}
		|\Omega(n,q)|=\left(1+o(1)\right)2^{2-n}\binom{q}{n}
	\end{equation*}
	when $n\equiv 0\Mod{4}$, and
	\begin{equation*}
		|\tilde{\Omega}(n+1,q)|=\left(1+o(1)\right)2^{1-n}\binom{q}{n}
	\end{equation*}
	when $n\equiv 3\Mod{4}$.\par 
	If $q\equiv 1\Mod{4}$ and assume that $n=\omega(1)$ and $n=o(q^{1/4})$, then we have
	\begin{equation*}
		|\Omega(n,q)|=\left(1+o(1)\right)2^{2-n}\binom{q}{n}
	\end{equation*}
	when $n$ is even, and
	\begin{equation*}
		|\tilde{\Omega}(n+1,q)|=\left(1+o(1)\right)2^{1-n}\binom{q}{n}
	\end{equation*}
	when $n$ is odd.
	\begin{proof}
		This follows immediately from Theorem \ref{thm_number.of.even.induced.subgraphs.paley.undirected}, Theorem \ref{thm_number.of.even.induced.subgraphs.paley.directed} and Equations \eqref{eq_grs.degree.parity.even.order.paley}-\eqref{eq_egrs.odd.induced.odd.order.paley.tournament}.
	\end{proof}
\end{theorem}

\section*{Acknowledgment}
This work is supported by the National Natural Science Foundation of China (No.\ 12371337) and the Natural Science Foundation of Hunan Province (No.\ 2023RC1003).

\appendix

\section{Expectations in random (di)graphs}\label{appendix.a}
For random graphs in $G_{n,p}$, viz.\,the graphs on the vertex set $[n]$ in which each pair of vertices are joined by an edge independently with probability $0<p<1$, we determine the expectation of the number of even induced subgraphs of any order $r$ of $G\in G_{n,p}$.
\begin{lemma}\label{lemma_asymptotic.of.expectation}
	If $r\to\infty$, we have
	\begin{equation*}
		\sum_{k=0}^{r}\binom{r}{k}(1-2p)^{k(r-k)}=2+o(1).
	\end{equation*}
	\begin{proof}
		If $p=1/2$ we certainly have $\sum_{k=0}^{r}\binom{r}{k}(1-2p)^{k(r-k)}=2$. Henceforth we assume $p\neq 1/2$. Let $\gamma$ denote $1-2p$. Let $0<\delta<1/2$ be a constant. Then 
		\begin{equation*}
			\sum_{k=0}^{r}\binom{r}{k}\gamma^{k(r-k)}=\sum_{k=0}^{\lfloor\delta r\rfloor}\binom{r}{k}\gamma^{k(r-k)}+\sum_{k=\lfloor\delta r\rfloor+1}^{\lceil(1-\delta)r\rceil-1}\binom{r}{k}\gamma^{k(r-k)}+\sum_{k=\lceil(1-\delta)r\rceil}^{r}\binom{r}{k}\gamma^{k(r-k)},
		\end{equation*}
		in which
		\begin{equation*}
			\left|\sum_{k=\lfloor\delta r\rfloor+1}^{\lceil(1-\delta)r\rceil-1}\binom{r}{k}\gamma^{k(r-k)}\right|\leq\sum_{k=\lfloor\delta r\rfloor+1}^{\lceil(1-\delta)r\rceil-1}\binom{r}{k}|\gamma|^{\delta^{2}r^{2}}\leq(2|\gamma|^{\delta^{2}r})^{r}=o(1).
		\end{equation*}\par 
		By Lemma \ref{lemma_binomial.coefficients.linear.size}, if $k\leq\delta r$ we have $\binom{r}{k}\leq\binom{r}{\delta r}\leq 2^{(1+\varepsilon)H(\delta)r}$ for some constant $\varepsilon>0$. We can choose a constant $M$ such that $2^{(1+\varepsilon)H(\delta)}|\gamma|^{M}<1$. Thus
		\begin{equation*}
			\begin{split}
				\left|\sum_{k=1}^{\lfloor\delta r\rfloor}\binom{r}{k}\gamma^{k(r-k)}\right|&\leq\sum_{k=1}^{\lfloor M\rfloor}\binom{r}{k}|\gamma|^{k(r-k)}+\sum_{k=\lfloor M\rfloor+1}^{\lfloor\delta r\rfloor}\binom{r}{k}|\gamma|^{k(r-k)}\\
				&\leq\sum_{k=1}^{\lfloor M\rfloor}\left(1+o(1)\right)\frac{(r|\gamma|^{r})^{k}}{k!|\gamma|^{k^{2}}}+\sum_{k=\lfloor M\rfloor+1}^{\lfloor\delta r\rfloor}2^{\left(1+\varepsilon\right)H(\delta)r}|\gamma|^{M(r-M)}\\
				&\leq o(1)+\frac{\delta r(2^{(1+\varepsilon)H(\delta)}|\gamma|^{M})^{r}}{|\gamma|^{M^{2}}}\\
				&=o(1).
			\end{split}
		\end{equation*}
		By symmetry, we have $\left|\sum_{k=\lceil(1-\delta)r\rceil}^{r-1}\binom{r}{k}\gamma^{k(r-k)}\right|=o(1)$ as well, for which the conclusion holds.
	\end{proof}
\end{lemma}
Let $X_{r}$ be number of even induced subgraphs of order $r$ in $G\in G_{n,p}$. The asymptotic formula $\mathbb{E}(X_{r})=\left(1+o(1)\right)\binom{n}{r}2^{-r+1}$ in the following result has been obtained in \cite[Lemma 2.4]{ferber2023subgraphs} with more general congruence conditions, by considering principal submatrices of the random matrix. We derive the exact formula for $\mathbb{E}(X_{r})$ using the indicator variable.
\begin{proposition}\label{prop_expectation.of.number.of.even.induced.subgraphs.undirected}
	We have
	\begin{equation*}
		\mathbb{E}(X_{r})=\frac{1}{2^{r}}\binom{n}{r}\sum_{k=0}^{r}\binom{r}{k}(1-2p)^{k(r-k)}=\left(1+o(1)\right)\binom{n}{r}2^{-r+1}.
	\end{equation*}
	In particular, if $p=1/2$, then $\mathbb{E}(X_{r})=2^{1-r}\binom{n}{r}$.
	\begin{proof}
		For each non-empty set $S\subseteq[n]$, define the indicator variable
		\begin{equation*}
			\mathds{1}_{S}=\begin{dcases}
				1 & \text{if $G[S]$ is even},\\
				0 & \text{otherwise}.
			\end{dcases}
		\end{equation*}
		Then $X_{r}=\sum_{S\subseteq[n]\atop|S|=r}\mathds{1}_{S}$, which yields
		\begin{equation}\label{eq_expectation.of.number.of.even.induced.subgraphs.indicator}
			\mathbb{E}(X_{r})=\sum_{S\subseteq[n]\atop|S|=r}\mathbb{P}(\mathds{1}_{S}=1).
		\end{equation}
		Observe that $\mathds{1}_{S}=\prod_{v\in S}\left(1+(-1)^{\deg_{G[S]}(v)}\right)/2$. This implies
		\begin{equation*}
			\begin{split}
				\mathbb{P}(\mathds{1}_{S}=1)&=\mathbb{E}\left(\prod_{v\in S}\frac{1+(-1)^{\deg_{G[S]}(v)}}{2}\right)=\frac{1}{2^{r}}\mathbb{E}\left(\sum_{T\subseteq S}\prod_{v\in T}(-1)^{\deg_{G[S]}(v)}\right)\\
				&=\frac{1}{2^{r}}\sum_{T\subseteq S}\mathbb{E}\left((-1)^{2e(T)+e(T,S\backslash T)}\right)=\frac{1}{2^{r}}\sum_{T\subseteq S}\mathbb{E}\left((-1)^{e(T,S\backslash T)}\right).
			\end{split}
		\end{equation*}
		Since the edges are chosen independently and the contribution of each possible edge in the product $\prod_{e\in e(T,S\backslash T)}(-1)$ equals $1-2p$, this probability equals
		\begin{equation}\label{eq_undirected.probability.of.G[S].even}
			\mathbb{P}(\mathds{1}_{S}=1)=\frac{1}{2^{r}}\sum_{T\subseteq S}(1-2p)^{|T|(|S|-|T|)}=\frac{1}{2^{r}}\sum_{k=0}^{r}\binom{r}{k}(1-2p)^{k(r-k)}.
		\end{equation}
		The Equations \eqref{eq_expectation.of.number.of.even.induced.subgraphs.indicator} and \eqref{eq_undirected.probability.of.G[S].even} then complete the proof, with the asymptotic approximation following from Lemma \ref{lemma_asymptotic.of.expectation}.
	\end{proof}
\end{proposition}
Random digraphs, denoted by $D_{n,p}$, are defined similarly as the digraphs on the vertex set $[n]$ where each pair of vertices $x,y$ are assigned $x\to y$ independently with probability $p$.
\begin{proposition}\label{prop_expectation.of.number.of.even.induced.subgraphs.directed}
	Let $\tilde{X}_{r}$ be the number of even induced subdigraphs of order $r$ in $D\in D_{n,p}$. Then
	\begin{equation*}
		\mathbb{E}(\tilde{X}_{r})=\frac{1}{2^{r}}\binom{n}{r}\left(1+(1-2p)^{r-1}\right)^{r}=\left(1+o(1)\right)\binom{n}{r}2^{-r}.
	\end{equation*}
	In particular, if $p=1/2$, then $\mathbb{E}(\tilde{X}_{r})=2^{-r}\binom{n}{r}$.
	\begin{proof}
		Analogously, we define
		\begin{equation*}
			\tilde{\mathds{1}}_{S}=\begin{dcases}
				1 & \text{if $D[S]$ is even},\\
				0 & \text{otherwise},
			\end{dcases}
		\end{equation*}
		for each non-empty set $S\subseteq[n]$ with $|S|=r$. Note that the out-degrees of vertices in $D[S]$ are mutually independent, which implies
		\begin{equation*}
			\mathbb{P}(\tilde{\mathds{1}}_{S}=1)=\frac{1}{2^{r}}\sum_{T\subseteq S}\mathbb{E}\left(\prod_{v\in T}(-1)^{\deg_{D[S]}^{-}(v)}\right)=\frac{1}{2^{r}}\sum_{T\subseteq S}\prod_{v\in T}\mathbb{E}\left((-1)^{\deg_{D[S]}^{-}(v)}\right).
		\end{equation*}
		Clearly, the expected value of $(-1)^{\deg_{D[S]}^{-}(v)}$ equals $(1-2p)^{|S|-1}$, thus
		\begin{equation*}
			\mathbb{P}(\tilde{\mathds{1}}_{S}=1)=\frac{1}{2^{r}}\sum_{T\subseteq S}(1-2p)^{|T|(|S|-1)}=\frac{1}{2^{r}}\sum_{k=0}^{r}\binom{r}{k}(1-2p)^{k(r-1)}=\frac{1}{2^{r}}\left(1+(1-2p)^{r-1}\right)^{r}.
		\end{equation*}
		This completes the proof since $\mathbb{E}(\tilde{X}_{r})=\sum_{S\subseteq[n]\atop|S|=r}\mathbb{P}(\tilde{\mathds{1}}_{S})$.
	\end{proof}
\end{proposition}


\bibliographystyle{amsalpha}
\bibliography{references}

\end{document}